\documentclass{lms}
\newtheorem{theorem}{Theorem}[section] % 1st argument is your name for it
\newtheorem{lemma}[theorem]{Lemma}     % 2nd argument is what is printed

\newtheorem{proposition}[theorem]{Proposition}
\usepackage{amsmath}
\usepackage{amssymb}
\usepackage{graphicx} 
\usepackage[all]{xy}
\newcommand{\R}{\mathbb{R}}

\newnumbered{assertion}{Assertion}    % 1st argument is your name for it
\newnumbered{conjecture}{Conjecture}  % 2nd argument is what is printed
\newnumbered{definition}{Definition}
\newnumbered{hypothesis}{Hypothesis}
\newnumbered{remark}{Remark}
\newnumbered{note}{Note}
\newnumbered{observation}{Observation}
\newnumbered{problem}{Problem}
\newnumbered{question}{Question}
\newnumbered{algorithm}{Algorithm}
\newnumbered{example}{Example}
\newunnumbered{notation}{Notation} % This is usually unnumbered
\title[On Bredon's trick]% end with percent
 {On Bredon's trick: from local to global properties} % This is the full title of the paper
\author{Mauricio Angel}

\dedication{Dedicated to F. Dalmagro}

\classno{54F99 (primary), 57N65,58A12, 55N33 (secondary).}

\journal{}

\begin{document}
\maketitle

\begin{abstract}
One of the most interesting problems that arise when studying certain structures on topological spaces and in particular on differential manifolds, is to be able to extend the properties that are valid locally to the whole space. A useful tool, which has perhaps been underestimated, is a lemma introduced by G. Bredon, which we refer to as Bredon's trick and which allows the extension of local properties to certain topological spaces. We make a review of this result and show its application in the context of De Rham's cohomology, we will see how this trick allows to give natural alternative demonstrations to classic results, as well as it is fundamental in other cases such as in stratified pseudo-manifolds.
\end{abstract}

\section{Introduction}

Bredon's trick is a lemma introduced by G. Bredon which allows to extend on a topological space with certain conditions properties that are locally valid to the whole space. In \cite{Bredon} Bredon refers to the fact that this lemma was originally introduced in 1962 while he was teaching a course on Lie Groups, however, as far as the author knows, there are no references where the lemma is found but in \cite{Bredon} where it is used to prove De Rham's theorem. From the Bredon's proof it is easy to conclude that it is a Mayer-Vietoris-style argument, but since the statement is sufficiently general, it would be more related to the Compactness theorem from Model Theory \cite{Paseau}, so it is natural to think that it is a tool with great potential for the demonstration of global results, from local arguments.

\smallskip

The local-to-global arguments are particularly useful in differential geometry, because they allow reducing the study of a certain property on a space, to the study in local terms which generally entails a minor difficulty. There are many examples \cite{Kobayashi} where this type of arguments can be seen, for example, recently in \cite{Nariman}, Thurston's theorem which identifies the homology of classifying spaces of homeomorphism on a manifold of dimension not greater than 3, was proved using an argument of the local-global type, which avoids the use of the theory of foliations.

\smallskip

One of the areas where this type of arguments can be very useful is in topological analysis of data, with the advent of the Big Data and the development of algorithms to process large amounts of data, one of the lines of study has to do with the dimensionality reduction so that the geometric structure of the data is preserved, however even with the various techniques to perform this reduction, in some cases is still a problem the handling of these data, so it requires more efficient algorithms, see for example \cite{Agnes}.

\smallskip

In this sense, it seems pertinent to us to make a complete review of this result, which has been used in several cases and yet has not been given the relevance it has, in addition to the potential it has as a tool for the study of properties on topological spaces and manifolds.In this paper we give a complete review to the Bredon's trick, in section \ref{preliminar} we have included concepts and preliminary results needed to make the test as clear and self-contained as possible. In section \ref{truco} we present Bredon's trick with the adapted proof of \cite{Mangel}. Section \ref{ejemplos} is dedicated to present some results in which the Bredon's trick has been used as a fundamental tool for its demonstration, although it is true that some are classic results, the last two results on stratified pseudo-manidfolds have more complexity and the use of Bredon's trick is fundamental for the proof.

\section{Preliminaries}\label{preliminar}

\begin{definition}
Given a topological space X, a cover $\{U_\alpha\}_{\in\Lambda}$ is said to be locally finite if for every $x\in X$ there is an open $U\subset X$ such that $x\in U$ and 
\[|\{\alpha\in\Lambda:\, U_\alpha\cap U\neq\emptyset\}|<\infty\]
\end{definition}

\begin{definition}
For two covers $\{V_\beta\},\, \{U_\alpha\}$ we say that $\{V_\beta\}$ is a refinement for $\{U_\alpha\}$ if for each $\beta$ there exists $\alpha$ such that
\[V_\beta\subset U_\alpha\]
\end{definition}

The concept of refinement induces a partial order on all the covers.

\begin{definition}
A topological space $X$ is called para-compact if every open cover has an open locally finite refinement.
\end{definition}

\begin{definition}
Given a continuous map $p:X\to\R$ we define the support of p, $Supp(p)$ by
\[Supp(p)=\overline{\{x\in X:\, p(x)\neq 0\}}\]
\end{definition}

\begin{definition}
For a topological space $X$ a partition of unity is a family of continuous functions $\{p_\alpha:X\to [0,1]\}_{\alpha\in\Lambda}$ such that:
\begin{enumerate}
\item The cover $\{Supp(p_\alpha)\}_{\alpha\in\Lambda}$ is locally finite.
\item $\displaystyle{\sum_{\alpha\in\Lambda}p_\alpha(x)=1.}$
\end{enumerate}
\end{definition}

The sum in the previous definition is defined because the locally finite condition on $\{Supp(p_\alpha)\}$

\begin{definition}
Given a cover $\{U_\alpha\}$ and a partition of unity $\{p_\alpha\}$, we say that the partition of unity is subordinate to the cover if for each $\alpha$ there exists $\beta$ such that
\[Supp(p_\alpha)\subset U_\beta\]
\end{definition}

\begin{proposition}
Given a para-compact space $X$ and an open cover $\{U_\alpha\}$, then  there is a partition of unity subordinate to the cover $\{U_\alpha\}$
\end{proposition}
\begin{proof}
From the cover $\{U_\alpha\}$ we take a open refinement locally finite $\{V_\beta\}$. As X is normal we can choose a open refinement $\{W_\gamma\}$ such that for every $\gamma$ there is a $\beta$ such that
\[\overline{W_\gamma}\subset V_\beta\]

By Urysohn' lemma, for every $\gamma$ there is a continuous map $f_\gamma:W_\gamma\to [0,1]$ such that
\[f_\gamma(x)=\begin{cases}1 & \text{if\:} x\in \overline{W_\gamma}\\ 0 &  \text{if\:} x\not\in V_\beta \end{cases}\]

Let us define the partition unity functions as:
\[p_\gamma(x)=\frac{f_\gamma(x)}{\sum_{\gamma}f_\gamma(x)}\]

Because $\{\overline{W_\gamma}\}$ is a cover for $X$, then there is at least one index $\gamma$ such that $f_\gamma(x)\neq 0$, and by the locally finite condition the sum in the denominator is well defined.
\end{proof}

\section{Bredon's trick}\label{truco}

We present in this section Bredon's trick, the demonstration of this result is based on the existence of a proper function, which allows the decomposition of the space into compact sets. In Bredon's proof this is included as a comment where he claims that it is possible to build this function, in this case we follow the presentation made by \cite{Mangel}. Let us recall that a continuous function is called proper if the inverse of a compact set is compact, the following result guarantee the existence of proper functions under certain conditions.

\begin{lemma}
Given a a topological space $X$ Hausdorff, second numerable and locally compact, then there is a continuous proper map $f:X\to [0,\infty)$
\end{lemma}

\begin{proof}
Let $\{U_n\}$ an open locally finite cover and $\{p_n\}$ a partition of  unity subordinate to this cover. We define $f:X\to [0,\infty)$ by
\[f(x)=\sum_n np_n(x)\]

Clearly $f(x)$ is continuous, let us see that is a proper map, for this is enough to prove 
\[f^{-1}[0, n]\subset \bigcup_{j=1}^n U_j\]

Let us suppose that there is $x\in X$ such that $0\leq f(x)\leq n$ and $x\not\in \bigcup_{j=1}^n U_j$. Let $\{U_{j_1},\cdots,U_{j_k}\}$ be the elements in the cover $\{U_n\}$ that have $x$ as element, we can assume that $j_1<\cdots <j_k$, then we have $n<j_1$ but as $\sum_{n}p_n(x)=p_{j_1}(x)+\cdots +p_{j_k}(x)=1$, then:
\[j_1=j_1(p_{j_1}(x)+\cdots +p_{j_k}(x))=\sum_{i=1}^n j_1p_{j_i}(x)\]
As $j_1<\cdots <j_k$ we have
\[j_1=\sum_{i=1}^n j_1p_{j_i}(x)\leq \sum_{i=1}^n j_ip_i(x)=\sum_{n} np_n(x)=f(x) < n\]

And this is a contradiction.
\end{proof}

\begin{lemma}
Let $X$ be a para-compact topological space and let $\{U_\alpha\}$ be an open covering, closed for finite intersection. Suppose that $p(U)$ is a statement about open subsets of $X$, satisfying the following three properties:
\begin{enumerate}
\item $p(U_\alpha)$ is true for each $\alpha$;
\item if $p(U)$, $p(V)$ and $p(U\cap V)$ is true, then $p(U\cup V)$ is also true, where $U$ and $V$ are open subsets of $X$;
\item let $\{U_i\}$ be a disjoint family of open subsets of $X$, if $p(U_i)$ is true then $p(\cup_i U_i)$ is also true.
\end{enumerate}
Then $p(X)$ is true.
\end{lemma}

\begin{proof}
The proof is performed into two stages, first it is proved that P is true for open finite unions by induction in the number of open sets, suppose that $P(U_i)$ is true for each $i$:
\begin{itemize}
\item $P(U_1\cup U_2)$ is condition 2).

\item If $P(U_1\cup U_2\cup\cdots\cup U_n)$ is true, from the identity 
\[(U_1\cup U_2\cup\cdots\cup U_n)\cap U_{n+1}=(U_1\cap U_{n+1})\cup (U_2\cap U_{n+1})\cup\cdots\cup (U_n\cap U_{n+1})\]
we obtain an union of $n$ open sets, then $P((U_1\cup U_2\cup\cdots\cup U_n)\cap U_{n+1})$ is true and again by condition 2) we have that $P((U_1\cup U_2\cup\cdots\cup U_n)\cup U_{n+1})$ is true.
\end{itemize}

In the second part we choose a proper function $f:M\to [0,\infty)$ and define sets $A_n=f^{-1}[n,n+1]$. Since $f$ is proper, $A_n$ is compact, so we can cover it with a finite union of sets $U_n$ contained in $f^{-1}(n-\frac{1}{2},n+\frac{3}{2})$, so we have
\[A_n\subset U_n\subset f^{-1}(n-\frac{1}{2},n+\frac{3}{2})\]

From here we get that $A_n$ with $n$ even are disjointed from each other, and if $n$ is odd, they are also disjointed from each other, besides as $A_n$ is the finite union of open $U_n$, for the first part, we have that $P(A_n)$ is true for all $n$.

\smallskip

Finally we write to $M$ as a union of two open sets:
\[U=\bigcup_{k\geq o}A_{2k}\: \text{and}\: V=\bigcup_{k\geq o}A_{2k+1}\]

Condition 3) guarantees that $P(U)$ and $P(V)$ are true, and since
\[U\cap V=\bigcup_{i,j}(A_{2i}\cap A_{2j+1})\]
is a union of open disjoints, and each $(A_{2i}\cap A_{2j+1})$ is a finite union, then $P((A_{2i}\cap A_{2j+1}))$ is true and then $P(U\cap V)$ is true, again applying condition 2) establishes that $P(U\cup V)=P(M)$ is true.
\end{proof}

\section{Application to Cohomology of Manifolds}\label{ejemplos}

We present now the application of Bredon's trick in obtaining results for De Rham's cohomology for differential manifolds, we seek with this to illustrate the power that this result has to prove global results from the local study on space.

\smallskip

Given a differentiable manifold $M$ we know that the differential forms induce a complex of cochains, whose homology is known as the De Rham's cohomology, it is also possible to define over $M$ the simplicial cochains complex that produces the singular cohomology of  $M$ . The equivalence of both cohomologies was established by De Rham in his thesis, and it uses tools of differential form calculus. Originally the introduction of Bredon's trick was made as an auxiliary lemma to give an ingenious demonstration of De Rham's theorem \cite{Bredon}, this same approach was presented in \cite{Marco} were some clarifications were made.

\begin{theorem}[(De Rham theorem)]
Let $M$ be a differentiable manifold, and let us denote by $H_{DR}(M)$ the De Rham cohomology and by $H_\sigma(M)$ the singular cohomologym then
\[H_{DR}^k(M)\cong H_\sigma^k(M).\]
\end{theorem}

\begin{proof}
The proof it is based on the Bredon's trick, for each open set $U\subset M$ we consider the property:
\[P(U):=\: "H_{DR}^k(U)\cong H_\sigma^k(U)"\]

We can choose a basis $\cal{B}$ of geodesic convex contained in the charts of $M$ which is closed by finite intersections. Then $P(U)$ is true for every $U\in \cal{B}$ by the Poincar\'e lemma, then we have condition 1).

\smallskip

Let us assume that $P( U),\, P(V)$ and $P(U\cap V)$ are true for open $U,\, V$
as both De Rham and singular cohomology sastify Mayer-Vietoris, we have the commutative diagram

\[\scalebox{0.8}[0.9]{
\xymatrix{H_{DR}^{k-1}(U)\oplus H_{DR}^{k-1}(V)\ar[r] \ar[d]^{\cong} &H_{DR}^{k-1}(U\cap V)\ar[r] \ar[d]^{\cong} & H_{DR}^{k}(U\cup V)\ar[r]\ar[d]&H_{DR}^{k}(U)\oplus H_{DR}^{k}(V)\ar[r] \ar[d]^{\cong} &H_{DR}^{k}(U\cap V)\ar[r] \ar[d]^{\cong} &\\
H_\sigma^{k-1}(U)\oplus H_\sigma^{k-1}(V) \ar[r]  & H_\sigma^{k-1}(U\cap V)\ar[r] &H_\sigma^{k}(U\cup V) \ar[r] &H_\sigma^{k}(U)\oplus H_\sigma^{k}(V) \ar[r]  &H_\sigma^{k}(U\cap V)\ar[r] &}}
\]

By the five lema, we have that $P(U\cup V)$ is true, this proves the second condition.

\smallskip

Let us suppose now that we have a open disjoint collection $\{U_\alpha\}_{\alpha\in\Lambda}$ such that $P(U_\alpha)$ is true for each $\alpha$, then the third condition of Bredon's trick follows from the fact that
\[Hom\left(\bigoplus_{\alpha\in\Lambda}\tilde{S}_q(U_\alpha),\R\right)\cong \bigoplus_{\alpha\in\Lambda}Hom(\tilde{S}_q(U_\alpha),\R)\]

\[\Omega^k\left(\bigcup_{\alpha\in\Lambda}U_\alpha\right)\cong\bigoplus_{\alpha\in\Lambda}\Omega^k(U_\alpha)\]

Where $\tilde{S}_q(U_\alpha)$ denotes the set of differentiable $q$-simplex on $U_\alpha$, then  we have:
\[H_{DR}^k(\bigcup_{\alpha\in\Lambda}U_\alpha)\cong \bigoplus_{\alpha\in\Lambda}H_{DR}^k(U_\alpha)\cong \bigoplus_{\alpha\in\Lambda}H_{\sigma}^k (U_\alpha)\cong H_{\sigma}^k (\bigcup_{\alpha\in\Lambda}U_\alpha)\]

\end{proof}

Recall that given a Lie group $G$ acting on a manifold $M$, for every $g\in G$ we have the map $g:M\to M$ given by $g(m)=g.m$ this map induce a map on the differential forms space
\[g^*:\Omega^k(M)\to\Omega^k(M)\]

\begin{definition}
We say that a $k$-form $\omega\in \Omega^k(M)$ is invariant if $g^*(\omega)=\omega$. We denote by $I\Omega^k(M)$ the space of invariant $k$-forms and by $IH^k(M)$ the corresponding cohomology groups.
\end{definition}

The following is a classical result that related the invariant cohomology groups with the De Rham cohomology, in this case we use the Bredon's trick to prove it when the groups $G=S^1$ acts freely, this appoach is discussed in \cite{Expedito}.

\begin{theorem}\label{Exp}
Let $M$ be a $S^1$-manifold with $S^1$ acting freely, then for every $k$, the inclusion map
1\[i : I\Omega^k(M)\to\Omega^k(M)\]
induce an isomorphism on cohomology.
\end{theorem}

\begin{proof}
The proof it is based on the Bredon's trick, for each open set $U\subset M/S^{1}$ we consider the property:
\[P(U):=\: "I\Omega^k(\pi^{-1}(U))\to\Omega^k(\pi^{-1}(U))\: \text{induces isomorphism in cohomology}"\]

Because $S^1$ acts freely on $M$, the quotient $M/S^1$ is a manifold, so we can choose a cover $\{U_\alpha\}$ closed by finite intersections, and we have $\pi^{-1}(U_\alpha)=S^1\times U$ for a open set $U\sim \R^n$
By Poincaré lemma, we have
\[\Omega^k(S^1\times U)\cong \Omega^k(S^1)\]
and
\[I\Omega^k(S^1\times U)\cong I\Omega^k(S^1)\]

Then it is enough to proof that $I\Omega^k(S^1)\to \Omega^k(S^1)$ induces isomorphism in cohomology, but 
\[IH^k(S^1)=IH^k(S^1)=\begin{cases} \R & \text{if}\, k=0, 1.\\ 0 & \text{otherwise}\end{cases}\]

Then $P(U_\alpha)$ is true for every $U_\alpha$.

\smallskip

Let us assume that for open sets $U,\, V$ we have thar $P(U), P(V)$ and $P(U\cap V)$ is true, then using the Mayer-Vietoris sequence we have the following diagram:
\[\scalebox{0.85}{
\xymatrix{\ar[r]&IH^{k-1}(\pi^{-1}(U))\oplus IH^{k-1}(\pi^{-1}(V))\ar[r] \ar[d]^{\cong} & IH^{k-1}(\pi^{-1}(U\cap V))\ar[r] \ar[d]^{\cong} & IH^{k}(\pi^{-1}(U\cup V))\ar[r]\ar[d]&\\
\ar[r]&H^{k-1}(\pi^{-1}(U))\oplus H^{k-1}(\pi^{-1}(V)) \ar[r]  & H^{k-1}(\pi^{-1}(U\cap V))\ar[r] & H^{k}(\pi^{-1}(U\cup V)) \ar[r] &}}
\]

\[\scalebox{0.85}{
\xymatrix{\ar[r]&IH^{k}(\pi^{-1}(U))\oplus IH^{k}(\pi^{-1}(V))\ar[r] \ar[d]^{\cong} & IH^{k}(\pi^{-1}(U\cap V))\ar[r] \ar[d]^{\cong} & \\
\ar[r]&H^{k}(\pi^{-1}(U))\oplus H^{k}(\pi^{-1}(V)) \ar[r]  & H^{k}(\pi^{-1}(U\cap V))\ar[r] & }}
\]

By the five lemma, we conclude that $IH^{k}(\pi^{-1}(U\cup V))\cong H^{k}(\pi^{-1}(U\cup V))$, then $P(U\cup V)$ is true.

Let $\{U_\alpha\}$ a disjoint family such that for each $\alpha$ the $P(U_\alpha)$ is true, then:
\[IH^k(\pi^{-1}\left(\cup_\alpha U_\alpha\right))=IH^k\left(\cup_\alpha \pi^{-1}(U_\alpha)\right)=\oplus IH^k(\pi^{-1}(U_\alpha))\]
As $P(U_\alpha)$ is true, we have that $IH^k(\pi^{-1}(U_\alpha))\cong H^k(\pi^{-1}(U_\alpha))$ then

\[\oplus IH^k(\pi^{-1}(U_\alpha))\cong \oplus H^k(\pi^{-1}(U_\alpha))=H^k(\pi^{-1}\left(\cup_\alpha U_\alpha\right))\]

Then $P(\cup_\alpha U_\alpha)$ is true, and by Bredon's trick $P(M/S^1)$ is true.

\end{proof}

Another classic result in the context of De Rham's cohomology is the K\"unneth formula, which relates the cohomology of two manifolds $M$ and $F$ and the cohomology of the product manifold $M\times F$. According to the hypothesis that we have about $M$, $F$ and their cohomology groups the formula is written in different ways, in Bott-Tu's book \cite{Bott} this formula is written in a quite compact way, imposing conditions on the cohomology group of  $F$, here we follow the proof provided in \cite{Mangel}.

\begin{theorem}[(K\"unneth formula)]
Let $M$ and $F$ be two manifolds, such that $F$ has finite dimensional cohomology, then we have
\[H^*(M\times F)=H^*(M)\otimes H^*(F)\]
\end{theorem}

\begin{proof}
The proof it is based on the Bredon's trick, we fix the manifold $M$ and for each open set $U\subset F$ we consider the property:
\[P(U):=\: "H^*(M\times U)=H^*(M)\otimes H^*(U)"\]
Following as in theorem \ref{Exp} the idea is to use the fact that locally a manifold is $\R^n$, Poincar\'e lemma for the first condition, Mayer-Vietoris and the five lema for the second condition and for the third condition:
\[H^k(M\times \left(\cup_\alpha U_\alpha\right))=\oplus_\alpha H^k(M\times U_\alpha)\cong \oplus_\alpha H^k(M)\otimes H^k(U_\alpha)=H^k(M)\otimes H^k(\cup_\alpha U_\alpha)\]
\end{proof}

There are some others results where the use of the Bredon's trick is fundamental to be proven, a nice example can be found in Royo's thesis \cite{Royo} where adaptation for regular Riemannian flows are made, obtaining interesting results in this context of Gysin sequences. Another topic where Bredon's trick is particularly useful is in the theory of stratified spaces, the rest of the paper is devoted to show a couple of results.

\smallskip

Recall that a stratification for a paracompact space $X$ is a partition into disjonit manifolds $\{S_\alpha\}$   which is locally finite and we have
\[S_\alpha\cap\overline{S_\beta}\neq\emptyset\Leftrightarrow S_\alpha\subset \overline{S_\beta}\]

$S_\alpha$ is called strata of $X$ and we denote by $\Sigma_X$ the collection of all singular (non open) strata.

\begin{definition}
We say that a stratified space $X$  is a stratified pseudomanifold if locally each singular strata looks like a cone, that is, for each $S\in \Sigma_X$ and $x\in S$ there is an isomorphism $\varphi: U_x\to \R^n\times cL_S$ where:
\begin{itemize}
\item $U_x$ is a stratified open neighborhood of $x$.
\item $L_S$ is a compact stratified space, called the Link of $S$.
\item $\R^n\times cL_S$ S is endowed with the stratification
\[\{\R^n \times \{\nu\}\} \cup \{\R^n \times S_\alpha \times ]0, 1[ | S_\alpha \in L_S\}\]
\item $\varphi(x)=(0, \nu)$
\end{itemize}
\end{definition}

For a stratified pseudomanifold $X$. An unfolding of $X$ consists of a manifold $\tilde{X}$, a family of unfoldings of the links $\{\tilde{ L}\to^{L} L\}$ and a continuous surjective proper map $L:\tilde{X}\to X$ satisfyin a series of conditions that we will not discuss here, but they have to do with good behavior in regular strata and the conical structure of the singular strata. For a stratified pseudomanifold $X$ and an unfolding $(\tilde{X},L)$ a differential form $\omega\in \Omega(X-\Sigma)$ is called a Verona form if the following conditions holds
\begin{enumerate}
\item There is a form $\tilde{\omega}\in\Omega(\tilde{X})$ such that $\tilde{\omega}|_{\tilde{X}-\partial\tilde{X}}=L^{*}_X(\omega)$
\item There is a form $\omega_\Sigma\in \Omega(\Sigma)$ such that  $\tilde{\omega}|_{\partial\tilde{X}}=L^{*}_X(\omega_\Sigma)$
\end{enumerate}

We denote by $\Omega_v(M)$ the family of Verona forms and by $H_v(M)$ the corresponding cohomology. In \cite{Royo2} the De Rham theorem is proved for Verona forms:

\begin{theorem}
Suppose that $M$ is the unfolding for a stratified pseudomanifold $X$, then the cohomology for the Verona forms and the De Rham cohomology of $M$ are isomorphic.
\end{theorem}

\begin{proof}
The proof it is based on the Bredon's trick,for each open set $U\subset M$ we consider the property:
\[P(U):=\: "H^*( U)\cong H_v^*(U)"\]
\end{proof}

Finally, another result in which the power of Bredon's trick is used is found in \cite{Saralegui} and is related to stratified pseudomanifolds, it relates the intersection homology for a stratified pseudomanifold $X$ and the singular homology, we refers to \cite{Saralegui} for all the details.

\begin{theorem}
Let $X$ be a stratified pseudomanifold. Then:
\begin{itemize}
\item $ H^{\bar{p}}_{*}(X) = H_{*}(X - \Sigma_X)$ if $p < 0$;
\item  $H^{\bar{q}}_{*}(X) = H_{*} (X)$ if $q'geq t$ and $X$ is normal.
\end{itemize}
\end{theorem}

\begin{proof}
The proof is based in defining the following inclusion maps
\[I_X : S_{*}(X - \Sigma X)\to SC_{*}^{\bar{p}}(X)\]
\[J_X : SC_{*}^{\bar{p}}(X)\to S_{*}(X)\]

And apply the Bredon's trick to the property
\[P(U):\: "I_U\: \text{and}\, J_U\: \text{are quasi-isomorphism}"\]
\end{proof}

%\begin{acknowledgements}\label{ackref}
%The \verb"acknowledgements" environment may be used to acknowledge
%indebtedness to colleagues, host institutions and referees. Accounts
%of grants and financial support should be made as a footnote on the
%title page using the \verb"\extraline{}" command in the preamble.
%\end{acknowledgements}

\affiliationone{% in this example, two authors share an institution
   M. Angel\\
   Grupo de \'Algebra y L\'ogica,\\ Universidad Central de Venezuela\\
   Av.Los Ilustres, Caracas 1010, Venezuela.
   \email{mauricio.angel@ciens.ucv.ve}}
% Important: Do not put any empty line here.
% Use \affiliationthree{} for any address positioned under \affiliationone
% Use \affiliationfour{}  for any address positioned under \affiliationtwo
\affiliationthree{%
   Current address:\\
   Quito-Ecuador\\
   \email{mauricio.angel.mat@gmail.com}}

\begin{thebibliography}{9}% Replace 9 by 99 if 10 or more references
%
% Please note the use of "\and" between author names below
%
\bibitem{Mangel}
{\bibname M. Angel}, `Una aplicación del Truco de Bredon a la Cohomología de De Rham', Universidad Central de Venezuela (2002).
%
\bibitem{Bott}
 {\bibname R. Bott \and L.W. Tu},
 {\em Differential Forms in Algebraic Topology} ( Springer Verlag, New York, 1982).
%
\bibitem{Bredon}
 {\bibname G. Bredon},
 {\em Topology and Geometry} ( Springer Verlag, New York, 1993).
%
\bibitem{Expedito}
{\bibname E. Cede\~no}, `Cohomología Invariante por acciones libres de $S^1$', Universidad Central de Venezuela (2010).
%
\bibitem{Kobayashi}
 {\bibname
T. Kobayashi}, `From “Local”to “Global” -Beyond the Riemannian Geometry-', \emph{Kavli IPMU News } 25 (March 2014) 4-11.
%
\bibitem{Marco}
{\bibname M. A. P\'erez}, `Cohomolog\'{\i}a de los Espacios Proyectivos Complejos', Universidad Central de Venezuela (2007).
%
\bibitem{Nariman}
 {\bibname
S. Nariman}, `A local to global argument on low dimensional manifolds', \emph{Trans. Amer. Math. Soc. }373 (2020), 1307-1342.
%
\bibitem{Paseau}
 {\bibname
A. Paseau}, `Proofs of the Compactness Theorem', \emph{History and
Philosophy of Logic }31:1 (2010), 73-98.
%
\bibitem{Royo2}
 {\bibname J.I. Royo},
 `La sucesi\'on de Gysin',  Universidad del Pa\'{\i}s Vasco-Euskal Herriko Unibertsitatea (1999).
%
\bibitem{Royo}
 {\bibname J.I. Royo},
 `Estudio Cohomol\'ogico de Flujos Riemannianos',  Universidad del Pa\'{\i}s Vasco-Euskal Herriko Unibertsitatea (2004).
%
\bibitem{Saralegui}
{\bibname M.  Saralegui}, `De Rham Intersection Cohomology for General Perversities',
{\em Illinois Journal of Mathematics }49 (2005) 737--758.
%
\bibitem{Agnes}
 {\bibname
A. Vathy-Fogarassy \ and J. Abonyi}, `Local and global mappings of topology representing networks', \emph{Information Sciences }179 (21) (2009)3791 - 3803.

\end{thebibliography}
\end{document}